\documentclass[10pt]{article}
\usepackage{amsmath, amsthm, amssymb}
\usepackage{graphics}
\usepackage{pinlabel}
\usepackage{color}
\usepackage[colorlinks,citecolor=blue,linkcolor=red]{hyperref}
\usepackage{graphicx}
\usepackage{nicefrac}						
\usepackage{mathtools}
\usepackage{float}
\usepackage{tikz-cd}
\usepackage{enumerate}

\usepackage{secdot}
\usepackage{caption}

\usepackage{fullpage}
\usepackage{colonequals}
\allowdisplaybreaks

\newcommand{\defeq}{\colonequals}
\newtheorem{mydef}{Definition}
\newtheorem{thm}{Theorem}
\newtheorem*{mainthm*}{Main Theorem}

\newtheorem{prop}{Proposition}
\newtheorem{lemma}{Lemma}

\DeclareMathOperator{\Mod}{Mod}
\DeclareMathOperator{\PB}{PB}

\DeclareMathOperator{\B}{B}
\DeclareMathOperator{\PMod}{PMod}
\DeclareMathOperator{\Conf}{Conf}
\DeclareMathOperator{\PSL}{PSL}
\DeclareMathOperator{\diam}{diam}
\DeclareMathOperator{\id}{id}

\DeclareMathOperator{\rank}{rank}

\numberwithin{thm}{section} 
\numberwithin{prop}{section} 
\numberwithin{lemma}{section} 
\numberwithin{mydef}{section}

\title{Least dilatation of pure surface braids}
\author{Marissa Loving \\ with an appendix by Marissa Loving and Hugo Parlier}					
\date{}

\begin{document}

\maketitle 

\begin{abstract}
We study the minimal dilatation of pseudo-Anosov pure surface braids and provide upper and lower bounds as a function of genus and the number of punctures. For a fixed number of punctures, these bounds tend to infinity as the genus does. We also bound the dilatation of pseudo-Anosov pure surface braids away from zero and give a constant upper bound in the case of a sufficient number of punctures. 
\end{abstract}

\section{Introduction}

Let $S_{g,n}$ be a surface of genus $g \geq 2$ with $n \geq 1$ punctures and let $S_g = S_{g, 0}$. We define the {\bf mapping class group} of $S_{g,n}$, denoted by $\Mod(S_{g,n})$, to be the group of orientation preserving homeomorphisms of $S_{g,n}$ up to isotopy. The {\bf pure mapping class group} of $S_{g,n}$, denoted $\PMod(S_{g,n})$, is the subgroup of $\Mod(S_{g,n})$ that fixes each puncture pointwise. 

Consider the following short exact sequence $$1 \longrightarrow \ker({\mathcal Forget}) \longrightarrow \PMod(S_{g,n}) \longrightarrow \Mod(S_{g}) \longrightarrow 1,$$ where ${\mathcal Forget}: \PMod(S_{g,n}) \to Mod(S_g)$ is the forgetful map obtained by ``filling in" the $n$ punctures of $S_{g,n}$. The {\bf $n$-stranded pure braid group} of a surface of genus $g$ is defined as the kernel of this map ${\mathcal Forget}$, denoted $\PB_n(S_g).$ This is isomorphic to the fundamental group of the configuration space of ordered $n$-tuples of points on $S_g$; see Section \ref{subsection:surfacebraids} for further discussion.

Given a {\bf pseudo-Anosov} mapping class $f \in \PB_n(S_g)$ we denote its {\bf dilatation} by $\lambda(f)$ and its {\bf entropy} by $\log( \lambda(f))$, which is indeed the topological entropy of the pseudo-Anosov representative of $f$. In particular, we will be interested in the least entropy \[L(\PB_n(S_g)) \defeq \inf \{\log(\lambda(f)) \mid f \in \PB_n(S_g) \text{ is pseudo-Anosov}\}.\]

\begin{mainthm*} For a surface $S_{g,n}$ of genus $g \geq 2$ with $n \geq 1$ punctures there exist constants $c, c' > 0$ such that, \[c \log \left( \left \lceil \frac{\log g}{n} \right \rceil \right) + c \leq L(\PB_n(S_g)) \leq c' \log \left(\left \lceil \frac{g}{n} \right \rceil \right) + c'.\]
\end{mainthm*}

Explicit values for $c$ and $c'$ are obtained from the bounds given in Theorem \ref{thm:upperBound}, Theorem \ref{thm:lowerbound}, and Theorem \ref{thm:infty}.

To put the Main Theorem in context, we recall the results of Penner \cite{penner91} and Tsai \cite{tsai09}, which give bounds on the least entropy in the whole mapping class group (Penner for closed surface and Tsai for punctured surfaces). In particular, Penner's result shows that $L(\Mod(S_g))$ goes to $0$ as $g$ tends to infinity and Tsai's result shows that, for fixed genus $g$, $L(\Mod(S_{g,n}))$ goes to $0$ as $n$ tends to infinity. These both contrast sharply with the behavior of the least entropy in the pure surface braid group demonstrated by the Main Theorem, which shows that, for a fixed number of punctures $n$, $L(\PB_n(S_g))$ goes to infinity as $g$ tends to infinity. 

\begin{thm}[Penner] \label{thm:penner} For a surface $S_g$ of genus $g \geq 2$, \[ \frac{\log 2}{12g -12} \leq  L(\Mod(S_g)) \leq \frac{\log 11}{g}.\] \end{thm}

Penner's bounds have been improved by many authors; see Aaber--Dunfield \cite{aaber10}, Bauer \cite{bauer92},  Hironaka \cite{hironaka10}, Hironaka--Kin \cite{hirokin06}, and Kin--Takasawa \cite{kin13}. A version of Theorem \ref{thm:penner} was proved for the case of punctured surfaces by Tsai in \cite{tsai09}. 

\begin{thm}[Tsai] \label{thm:tsai} For any fixed $g \geq 2$, there is a constant $c_g \geq 1$ depending on $g$ such that \[ \frac{\log n}{c_g n} < L(\Mod(S_{g,n})) < \frac{c_g \log n}{n},\] for all $n \geq 3$. \end{thm} 

The constants $c_g$ in Tsai's result were improved from an exponential dependence on genus to a polynomial one by Yazdi \cite{yazdi17}.  

In addition to studying the least entropy of the mapping class group, many people have studied the least entropy of various subgroups of the mapping class group. For example, Farb--Leininger--Margalit studied the minimal entropy of the Torelli group, the Johnson kernel, and congruence subgroups in \cite{FLM08} and Hirose--Kin studied the least entropy of hyperelliptic handlebody groups in \cite{kin17}. The least entropy of classical pure braid groups, that is the fundamental group of the configuration space of ordered $n$-tuples of points on a disc, has also been an object of significant study. Song provided upper and lower bounds for the least entropy of the classical braid groups in \cite{song05}. Specific values of the least entropy were found when $n=4$ and $n=5$ by Song--Ko--Los \cite{songkolos02} and Ham--Song \cite{HamSong07}, respectively. More recently, Lanneau--Thiffeault \cite{lanneau11} gave simple constructions to realize the least entropy for $n = 4, 5$ and found the least entropy for braid groups of up to $8$-strands. 

The entropy of pseudo-Anosovs in the point pushing subgroup was also studied extensively by Dowdall in \cite{dowdall11}. Note that the point pushing subgroup coincides with the 1-stranded pure surface braid group $\PB_1(S_g)$. Combining the upper bound of Aougab and Taylor \cite{aougab15} and the lower bound of Dowdall \cite{dowdall11} gives the following.

\begin{thm}[Aougab--Taylor, Dowdall] \label{thm:dat} For the closed surface $S_g$ of genus $g \geq 2$, \[\frac{1}{5} \log(2g) \leq L(\PB_1(S_g)) < 4\log(g)  + 2 \log(24).\]\end{thm}

For fixed genus, the upper bound in our Main Theorem interpolates between the $\log(g)$ upper bound in Theorem \ref{thm:dat} in the case of a single puncture and a constant upper bound of $4\log(6)$ when $n > 2g$; see Theorem \ref{thm:upperBound}. 

Dilatations of pseudo-Anosov mapping classes have been studied in a number of other situations; see \cite{minakawa06, bai16, putman16, shinstrenner15, pankau17}. In fact, an analagous problem to ours on small dilatation pseudo-Anosovs has been studied in the context of nonorientable surfaces by Strenner \cite{strenner18}. 

\subsection*{Outline of the Paper} Section \ref{section:prelim} contains a brief introduction to surface homeomorphisms, surface braids, and Thurston's construction for psuedo-Anosovs; it also recalls several important results from quasiconformal geometry. In Section \ref{section:pointpushing} we give the details of a construction of Aougab--Taylor from \cite{aougab15} of a small dilatation psuedo-Anosov in the point pushing subgroup. Section \ref{section:upperbounds} outlines the construction of pseudo-Anosov pure surface braids realizing the upper bounds given in the Main Theorem. Section \ref{section:lowerbound} contains the proof of the constant lower bound implied by our Main Theorem. In Section  \ref{section:anotherlowerbound} we prove our Main Theorem's lower bound with explicit constants computed. We end the paper with an appendix which provides a lower bound on the diameter of a ``filling" graph embedded in a surface, a result which is needed in Section \ref{section:anotherlowerbound}. 

\subsection*{Acknowledgment} This material is based upon work supported by the National Science Foundation Graduate Research Fellowship under Grant No. DGE 1144245. The author would like to thank her PhD advisor Christopher Leininger for his guidance, insights, and endless encouragement. 

\section{Preliminaries} 
\label{section:prelim}

Here we establish our notation for the remainder of the paper and recall the necessary notions, definitions, and tools. In particular, we will give an overview of surface homeomorphisms, surface braids, some relevant results from quasiconformal geometry, and the details of a construction of psuedo-Anosovs due to Thurston. 

\subsection{Surface Homeomorphisms}
\label{subsection:surfacehomeos}

Let $S$ be a connected, oriented surface of genus $g$ possibly with a finite number of punctures and let $f: S \to S$ be a homeomorphism. Throughout the rest of the paper we will assume any surface we discuss is as described here.  

The homeomorphism $f$ is called {\bf pseudo-Anosov} if there exists a pair of transverse measured foliations $(\mathcal F^{s}, \mu_s)$ and $(\mathcal F^{u}, \mu_u)$ on $S$ and a real number $\lambda(f) > 1$ such that \[f \cdot (\mathcal F^{s}, \mu_{s}) = (\mathcal F^{s}, \lambda(f)^{-1} \mu_{s}) \text { and } f \cdot (\mathcal F^{u}, \mu_{u}) = (\mathcal F^{u}, \lambda(f) \mu_{u}).\] We call $\lambda(f)$ the {\bf stretch factor} or {\bf dilatation} of $f$. 

If there is a collection $\mathcal C$ of disjoint, essential simple closed curves on $S$ such that the homeomorphism $f$ preserves $\mathcal C$, then $f$ is said to be {\bf reducible}. If there is some power of the homeomorphism $f$ isotopic to the identity, then $f$ is called {\bf periodic} or {\bf finite order}.

A mapping class $\varphi \in \Mod(S)$ is said to be {\bf pseudo-Anosov}, {\bf reducible}, or {\bf periodic}, respectively, if there is a representative homeomorphism $f \in \varphi$ such that $f$ is {\bf pseudo-Anosov}, {\bf reducible}, or {\bf periodic}, respectively. Thurston proved the following classification of elements in $\Mod(S)$.

\begin{thm}[Nielsen--Thurston] A mapping class $\varphi \in \Mod(S)$ is pseudo-Anosov, reducible, or periodic. In addition, $\varphi$ is pseudo-Anosov if and only if it is neither reducible nor periodic. \end{thm}

A proof of this result can be found in \cite{FLP12}, as well as a detailed discussion of the definitions above. The interested reader can also find an introduction to these topics in \cite{farbmarg}.

\subsection{Surface Braids}
\label{subsection:surfacebraids}

Let X be a topological space. We define the {\bf configuration space} of $n$ distinct ordered points in $X$ to be the subspace of $X^n$ given by \[\Conf(X, n) \defeq \{(x_1, x_2, \ldots, x_n) \, : \, x_i \neq x_j \text{ for } i \neq j\}.\] Note that the symmetric group, $\Sigma_n$, acts on $\Conf(X,n)$ on the left by \[\sigma(x_1, \ldots, x_n) = (x_{\sigma(1)}, \ldots x_{\sigma(n)}).\]

\begin{mydef} \label{mydef:braids} Let $S$ be a surface. The {\bf braid group} of $S$ on $n$-strands is \[\B_n(S) \defeq \pi_1(\Conf(S, n) /\Sigma_n).\] The {\bf pure braid group} of $S$ on $n$-strands is \[\PB_n(S) \defeq \pi_1(\Conf(S,n)).\] \end{mydef}

Note that $\B_n = \pi_1(\Conf(\mathbb C, n)/\Sigma_n)$ and $\PB_n = \pi_1(\Conf(\mathbb C, n))$ are the classical braid and pure braid groups, respectively; see \cite{farbmarg}. Although at first glance Definition \ref{mydef:braids} appears different from the definition of $\PB_n(S_g)$ given in the introduction, Birman established that these definitions are equivalent in the following theorem, which first appeared in \cite{birman69}.  

\begin{thm}[Birman] For each pair of integers $g,n \geq 0$ let ${\mathcal Forget}: \PMod(S_{g,n}) \to \Mod(S_g)$ be the forgetful map. If $g \geq 2$, then $\ker({\mathcal Forget})$ is isomorphic to $\pi_1(\Conf(S_g,n)).$ \end{thm}

The proof of this result appeals to a long exact sequence of homotopy groups and is not immediately obvious, but the intuition is straightforward. Observe that for a homeomorphism representing a mapping class in the $n$-stranded pure braid group the isotopy on the closed surface from the homeomorphism back to the identity traces out a loop of $n$ ordered point configurations and this defines the isomorphism. A further discussion of braid groups can be found in \cite{birman74}. 

\subsection{Some Quasiconformal Results}
\label{subsection:quasiconformal}

A Teichm\"uller theoretic approach is employed in the proof of Theorem \ref{thm:infty}, which is part of the lower bound in the Main Theorem. Consequently, we will need some results from quasiconformal geometry. We begin by defining a quasiconformal map; see \cite{ahlfors06} for more on quasiconformal mappings.

\begin{mydef}Let $f: \Omega \to f(\Omega)$ be a homeomorphism between open sets $\Omega, f(\Omega) \subset \mathbb C$. Suppose $f$ has locally integrable weak partial derivatives and let $\displaystyle D_f = \frac{|f_z| + |f_{\bar z}|}{|f_z| - |f_{\bar z}|} \geq 1.$ We say that $f$ is {\bf quasiconformal} if $\left \Vert D_f \right \Vert_{\infty} < \infty$ and {\bf $K$-quasiconformal} if $\left \Vert D_f \right \Vert_{\infty} \leq K$. The {\bf quasiconformal dilatation} is $K(f) = \left \Vert D_f \right \Vert_{\infty}.$\end{mydef} 

An important component of our proof is a result of Teichm\"uller \cite{teich42} and Gehring \cite{gehring70} which relates the dilatation of a quasiconformal map $f$ on the hyperbolic plane $\mathbb H^2$ to the maximum distance which any point of $\mathbb H^2$ is moved by $f$. We give a version of the statement which can be found in Kra \cite{kra81}.

\begin{thm}[Kra] \label{thm:gehring} Consider $\mathbb H^2$ with Poincar\'e metric $\rho$. For $x, y \in \mathbb H^2$ there exists a unique self-mapping $f: \mathbb H^2 \to \mathbb H^2$ so that $f$ is the identity on the boundary of $\mathbb H^2$, $f(x) = y$, and $f$ minimizes the quasiconformal dilatation among all such mappings. Let $K(x,y)$ be the quasiconformal dilatation of such an extremal $f$. Then there exists a strictly increasing real-valued function $\varkappa: [0, \infty) \to [0, \infty)$ such that 
\begin{enumerate}[(i)]
\item $\log(1 + \frac{t}{2}) \leq \varkappa(t),$ and
\item $\frac{1}{2} \log K(x, y) = \varkappa(\rho(x,y)).$
\end{enumerate} \end{thm}

The second important component of the proof of Theorem \ref{thm:infty} is Theorem \ref{thm:canonical} below. The statement and proof of Theorem \ref{thm:canonical} in the case of $n = 2$ are due to Imayoshi--Ito--Yamamoto \cite{iiy03} with a weaker upper bound on the quasiconformal dilatation. The proof of Imayoshi--Ito--Yamamoto holds in the case of $n > 2$ punctures without any modification so we will omit the full argument and will instead provide a sketch of proof, namely the construction of $F_t$ and a justification of our improved upper bound on $K_t = K(F_t)$.  

\begin{thm}[Imayoshi--Ito-Yamamoto] \label{thm:canonical} Let $\varphi: S_{g,n} \to S_{g,n}$ be a pseudo-Anosov homeomorphism representing an element of $\PB_n(S_g)$ and let $\widehat{\varphi}:S_g \to S_g$ be the extension of $\varphi$ to the surface with the punctures filled in. There exists a conformal structure on $S_{g}$ together with an isotopy $F_t: S_g \to S_g$ with $t \in [0,1]$, through quasiconformal maps, between $\id:S_g \to S_g$ and $\widehat{\varphi}$ on the closed surface $S_g.$ Furthermore, for each $t \in [0,1]$ the quasiconformal dilatation $K_t$ of $F_t$ satisfies \[\log(K_t) \leq 3 \log(\lambda(\varphi)).\] \end{thm}

\begin{proof}[Sketch of Proof] We will begin by constructing $F_t$. Let $S_g$ be given a conformal structure so that $[id]=[id : S_{g,n} \to S_{g,n}]$ lies on the axis for $\varphi$ and $[0,1] \ni t \mapsto [f_t] \in \mathcal T(S_{g,n})$ be the Teichm\"uller geodesic connecting $[\id]$ and $\varphi^{-1}([\id])$. So for all $t \in [0,1]$, $f_t:S_{g,n} \to f_t(S_{g,n})$ is a Teichm\"uller mapping and \[\frac{1}{2} \log(K(f_t)) \leq \frac{1}{2} \log(K(f_1)) = \log(\lambda(\varphi^{-1})) = \log(\lambda(\varphi)).\] 
By filling in the punctures, we can extend $f_t$ to $\widehat{f_t}:S_g \to \widehat{f_t}(S_g).$ Denote by $\widehat{\varphi_t}$ the Teichm\"uller map of $S_g$ onto $\widehat{f_t}(S_g)$ isotopic to $\widehat{f_t}$ on $S_g$. Then we define the map $F_t: S_g \times [0,1] \to S_g$ by \[F_t(x) = \widehat{\varphi_t}^{-1} \circ \widehat{f_t}(x) \text{ for } x \in S_g \text{ and } t \in [0,1].\] The fact that $F_t$ is an isotopy is proved in \cite{iiy03}. Note that \[\log(K_t) = \log(K(\widehat{\varphi_t}^{-1} \circ \widehat{f_t})) \leq \log(K(\widehat{\varphi_t}^{-1})) + \log(K(\widehat{f_t})).\] Furthermore, we have that $t \mapsto [\widehat{f_t}]$ is a closed loop of length at most $\log(\lambda(\varphi))$. So \[\frac{1}{2}\log(K(\widehat{\varphi_t}^{-1})) = d_{\mathcal T(S_g)}([\widehat{f_t}], [\id]) \leq \diam_{\mathcal T(S_g)}(\{[\widehat f_s] \mid s \in [0,1]\}) \leq \frac{1}{2} \log(\lambda(\varphi)).\] Thus, \[ \log(K_t) \leq 3 \log(\lambda(\varphi)).\qedhere \] \end{proof}

\subsection{Thurston's Construction}
\label{subsection:thurston}

Here we will introduce a useful tool for constructing pseudo-Anosov mapping classes due to Thurston \cite{Thur88}. We say a collection $C$ of essential simple closed curves {\bf fills} our surface $S = S_{g,n}$ if the curves intersect transversely and minimally and the complement of $C$ in $S$ is a collection of disks and once-puncture disks. Equivalently, we could say that $C$ fills $S$ if any essential simple closed curve on $S$ has nonzero geometric intersection number with at least one curve in our collection $C$. 

Now suppose we have a collection $C = \{c_1, c_2, \ldots, c_m\}$ of pairwise disjoint, essential simple closed curves on $S$. We can define a {\bf multi-twist} $T_C$ about $C$ to be the product of positive Dehn twists about each $c_i \in C$. 

\begin{thm}[Thurston] \label{thm:ThurCons} Let $A = \{\alpha_1, \alpha_2, \ldots, \alpha_m\}$ and $B = \{\beta_1, \beta_2, \ldots, \beta_k\}$ be collections of pairwise disjoint, essential, simple closed curves on S such that $A \cup B$ fills $S$. There is a real number $\mu > 1$ and homomorphism \[\rho: \left< T_A, T_B \right> \to \PSL(2, \mathbb R) \text{ given by }\] \[T_A \mapsto\begin{pmatrix} 1 & - \mu^{1/2} \\ 0 & 1 \end{pmatrix} \text{ and } T_B \mapsto\begin{pmatrix} 1 & 0 \\ \mu^{1/2} & 1 \end{pmatrix}.\]  Furthermore, for $f \in \left<T_A, T_B\right>$, f is pseudo-Anosov if its image $\rho(f)$ is hyperbolic, in which case the dilatation of $f$ is equal to the spectral radius of $\rho(f).$\end{thm}

Consider a mapping class $T_{A} T_{B}^{-1} \in \left< T_A, T_B \right>$ as given by Theorem \ref{thm:ThurCons}. The image of $T_A T_B^{-1}$ under $\rho$ is given by \[\begin{pmatrix} 1 & - \mu^{1/2} \\ 0 & 1 \end{pmatrix}\begin{pmatrix} 1 & 0 \\ \mu^{1/2} & 1 \end{pmatrix}^{-1} = \begin{pmatrix} \mu + 1 & - \mu^{1/2} \\ - \mu^{1/2} & 1 \end{pmatrix}.\] The trace of this matrix is $2 + \mu$. Thus, by Theorem \ref{thm:ThurCons}, $T_{A}T_{B}^{-1}$ is pseudo-Anosov and $\log(\lambda(T_A T_B^{-1}))$ is bounded above by $\log(2 + \mu)$. 

The real number $\mu$ in Theorem \ref{thm:ThurCons} is the Perron--Frobenius eigenvalue of $NN^T$, where $N$ is defined as $N_{i,j} = i(\alpha_i, \beta_j).$ If $A = \{ \alpha \}$ and $B = \{ \beta \}$, then $\mu = i(\alpha, \beta)^2$. This will be a useful fact to keep in mind for the following section. In general, $\mu$ cannot be computed in such a straightforward manner. However, we can bound $\mu$ from above by the maximum row sum of $NN^T$; see \cite{gantmacher59}.

In order to compute the row sums of $NN^T$ we will follow the method used in \cite{ALM16}, which we describe here. Given $N$, we can build a labeled bipartite graph $G$ with $m$ red vertices and $k$ blue vertices corresponding to the multicurves $A$ and $B$, respectively. An edge from the $i$th red vertex to the $j$th blue vertex exists if $N_{i,j} \neq 0$, in which case it is labeled by $N_{i,j}$. We will define the {\bf weight} of a path in $G$ to be the product of edge labels in that path. The $(i, j)$ entry of $NN^T$ is equal to the sum of the weights of the paths of length 2 from the $i$th red vertex to the $j$th red vertex in $G$. To compute the row sum of $NN^T$ corresponding to a particular curve we start at the vertex associated to that curve and sum the weights of all paths of length two, possibly with backtracking, beginning at that vertex. 

\section{The Point Pushing Subgroup}
\label{section:pointpushing}

The construction given by Aougab--Taylor in \cite{aougab15} of point pushing homeomorphisms used to realize the upper bound in Theorem \ref{thm:dat} will play an important role in our proof of the Main Theorem so we will recall it here. We also employ some further work of Aougab--Huang \cite{aougabhuang14} to gain a more careful estimate of the upper bound than that provided in \cite{aougab15}. In particular, we will prove the following. 

\begin{thm}[Aougab--Taylor]\label{thm:aougabtaylor} For the closed surface $S_g$ of genus $g \geq 2$, \[L(\PB_1(S_g)) < 4\log(g)  + 2 \log(24).\]\end{thm}

\begin{proof}[Proof of Theorem \ref{thm:aougabtaylor}] Let $\alpha$ and $\beta$ be a minimally intersecting filling pair of curves on the closed surface $S_{g}$. By \cite{aougabhuang14}, we have that $i(\alpha, \beta) = 2g -1$. Let $\beta_1, \beta_2$ be the boundary components of a small tubular neighborhood of $\beta$. Thus, $\beta_1, \beta_2$ are homotopic to $\beta$ on $S_g$. Now place a marked point $z$ at some point of $\beta \setminus \alpha$. We can puncture $S_g$ at $z$ to form the surface $S_{g, 1}.$ 

Set $f_{\beta} = T_{\beta_1}^3 \circ T_{\beta_2}^{-3}$. This is a point pushing map in $S_{g,1}$ obtained by pushing the marked point $z$ along $\beta$ three times. Our goal is to show that $\{ \alpha, f_{\beta}(\alpha)\}$ fills the punctured surface $S_{g,1}$, and then apply Theorem \ref{thm:ThurCons} to obtain a pseudo-Anosov mapping class in $\PB_1(S_g)$. We apply the following inequality of Ivanov found in \cite{Ivanov92} to show that any essential simple closed curve on $S_{g,1}$ must intersect either $\alpha$ or $f_{\beta}(\alpha)$. 

\begin{lemma}[Ivanov] \label{lemma:ivanov} Let $c_1, \ldots c_m$ be a collection of pairwise disjoint, pairwise non-homotopic simple closed curves on a surface $S$ with negative Euler characteristic and let $(s_1, \ldots, s_m) \in \mathbb Z^m$. For any simple closed curves $\gamma, \rho,$ \begin{align*} \sum_{i = 1}^{m} (\left | s_i \right | - 2) i(\rho, c_i) i(c_i, \gamma) - i(\rho, \gamma) & \leq i(T_{c_1}^{s_1} \circ \cdots \circ T_{c_m}^{s_m}(\rho), \gamma) \\ & \leq \sum_{i =1}^{m} \left | s_i \right | i(\rho, c_i) i(c_i, \gamma) + i(\gamma, \rho). \end{align*} \end{lemma}

Suppose $\gamma$ is an essential simple closed curve on $S_{g,1}$ such that $i(\gamma, \alpha) = 0$. Now we can apply Lemma \ref{lemma:ivanov} with $\rho = \alpha$, $(s_1, s_2) = (3, -3)$, and $(c_1, c_2) = (\beta_1, \beta_2)$. Recall that $\alpha$ and $\beta$ filled $S_g$, so $\{ \alpha, \beta_1, \beta_2 \}$ fill $S_{g,1}$. Thus, $i(\gamma, \beta_i) \neq 0$ for $i = 1,2$, which implies that the lefthand side of the inequality in Lemma \ref{lemma:ivanov} is nonzero. Hence, $i(\gamma, f_{\beta}(\alpha)) \neq 0$, as desired. Furthermore, we can use the fact that $i(\alpha, \beta) = 2g -1$, together with Lemma \ref{lemma:ivanov}, to calculate that $i(\alpha, f_{\beta}(\alpha)) \leq 24g^2 - 24g + 6$.

Since $f_{\beta}$ is a point pushing map, we know that $\alpha$ and $f_{\beta}(\alpha)$ are homotopic on the closed surface $S_g$. Thus, $T_{\alpha} T_{f_{\beta}(\alpha)}^{-1} \in \PB_1(S_g)$ and, by Theorem \ref{thm:ThurCons}, is also pseudo-Anosov. Recall that in the case of two filling curves Theorem \ref{thm:ThurCons} tells us that the $\lambda(T_{\alpha} T_{f_{\beta}(\alpha)}^{-1}) \leq i(\alpha, f_{\beta}(\alpha))^2 + 2.$ Thus, $\lambda(T_{\alpha} T_{f_{\beta}(\alpha)}^{-1}) < 24^2 g^4$ and we obtain the desired upper bound \[ L(\PB_1(S_g)) < 4 \log(g) + 2 \log(24). \qedhere \] \end{proof}

As we will need this construction in the next section, we will denote the curves $\alpha$ and $f_{\beta}(\alpha)$ which we constructed above by $\alpha$ and $\tau$, respectively, and call them an {\bf Aougab--Taylor pair}. Note that we can construct an Aougab--Taylor pair $\{ \alpha, \tau \}$ on a surface of genus $g$ with a single boundary component with the same bound of $24g^2 - 24g + 6$ on intersection number, since on a surface of genus $g >2$ with a single boundary component there exists a pair of filling curves that intersect $2g-1$ times. In the case of a genus two surface with a single boundary component a minimally intersecting pair of filling curves will intersect $4$, not $3$, times. However we can still construct an Aougab--Taylor pair $\{ \alpha, \tau \}$ with $i(\alpha, \tau) \leq 24 .$ When our surface is a torus  with a single boundary component, we can construct an Aougab--Taylor pair $\{ \alpha, \tau \}$ with $i (\alpha, \tau) = 6$. 

\section{The Upper Bounds}
\label{section:upperbounds}

We will begin by proving the Main Theorem's upper bound which depends on the genus $g$ and number of punctures $n$ of our surface. We state this upper bound with explicit constants in Theorem \ref{thm:upperBound}. To prove the upper bound it suffices to construct a pseudo-Anosov pure braid satisfying the desired upper bound for each $g$ and $n$.  

\begin{thm} \label{thm:upperBound} For a surface $S_{g}$ of genus $g \geq 2$ with $1 \leq n \leq 2g$, we have \[L(\PB_n(S_g)) \leq 4 \log \left( \left \lceil \frac{2g}{n} \right \rceil \right) + 4 \log(7).\]\end{thm}

Fix a genus $g \geq 2$. Our main tool throughout this section will be leveraging Thurston's construction to build our desired pseudo-Anosov pure surface braids by building pairs of filling multicurves. 

\begin{proof}[Proof of Theorem \ref{thm:upperBound}] The main strategy of our proof is to divide our surface into subsurfaces with a single boundary component, fill each of these subsurfaces with an Aougab--Taylor pair, and then add a few additional curves which bound twice punctured disks to combine these Aougab--Taylor pairs into a single pair of filling multicurves. We will employ this strategy in each of our three cases: when $n = 2,3$, when $4 \leq n < 2g$, and when $n \geq 2g$.

\begin{proof}[Case 1] We begin our construction in the case of $n = 2$. Let $A$ and $B$ denote the multicurves marked in red and blue, respectively, in Figure \ref{fig:2_3punctures} which are constructed in the following way. Consider two subsurfaces of $S_{g,n}$ given by cutting along a separating curve that divides $S_{g,n}$ into two subsurfaces of genus at most $\left \lceil \frac{g}{2} \right \rceil$ each containing a single puncture. On each of these subsurfaces we can construct an Aougab--Taylor pair as described in Section \ref{section:pointpushing}. We then add an additional curve bounding a twice-punctured disk containing the pair of punctures. We illustrate this construction in Figure \ref{fig:2_3punctures} for the case of a genus 2 surface. In this situation our Aougab--Taylor pairs on each genus 1 subsurface intersect 6 times and our additional red curve, which bounds a twice-punctured disk containing the pair of punctures, intersects each blue curve 8 times. For $n=3$ we can add an additional puncture, as shown on the right of Figure \ref{fig:2_3punctures}.

\begin{figure}[H]
\centering
\includegraphics[width = 2.2in]{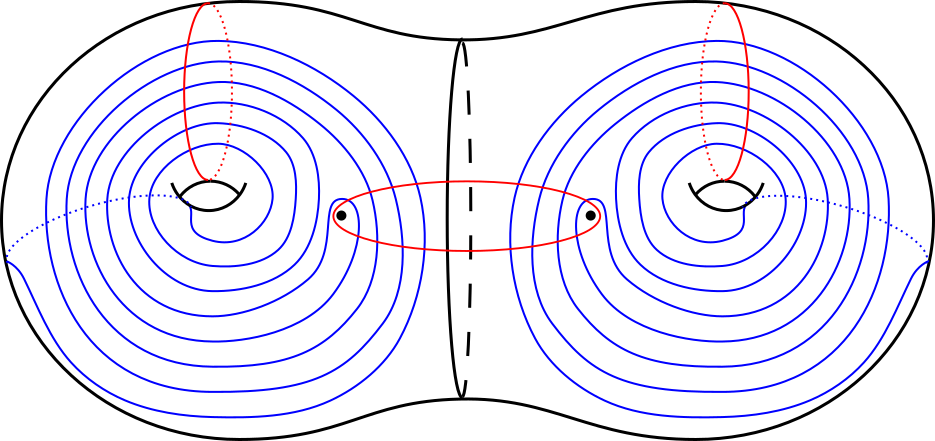} 
\hspace{5mm}
\includegraphics[width=2.2in]{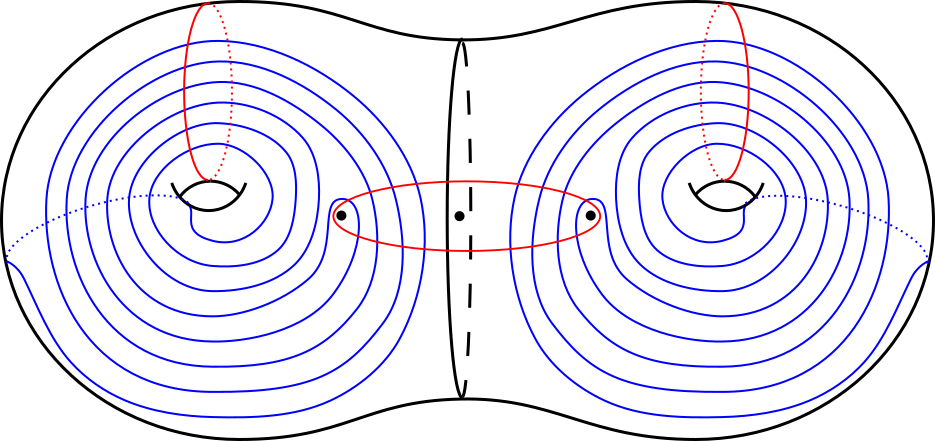}
\caption{Construction of filling multicurves, $A$ and $B$, for 2 and 3 punctures}
\label{fig:2_3punctures}
\end{figure}

Let $f = T_A T_B^{-1}$. Note that $f$ is pseudo-Anosov by Thurston's Construction, since $A$ and $B$ jointly fill $S_{g,n}$. Furthermore, $f \in \PB_n(S_g)$, since the red curve bounding the twice punctured disk is trivial on the closed surface and the pairs of curves which fill each subsurface will be homotopic to each other on the closed surface. Thus, the composition of positive and negative multitwist about $A$ and $B$ is the trivial mapping class on the closed surface. As discussed in Section \ref{subsection:thurston}, we can bound $\lambda(f)$ from above by the Perron--Frobenius eigenvalue, $\mu$, of $NN^T$. Since there are only $5$ curves in $A \cup B$ as shown in Figure \ref{fig:2_3punctures}, we can explicitly compute $\mu$. Note that the red curve which bounds a twice (or thrice) punctured disk intersects each blue curve at most $24 \left( \left \lceil \frac{g}{2} \right \rceil \right)^2 - 24\left \lceil \frac{g}{2} \right \rceil + 8$ times. So we have that $\mu \leq 3(24 \left( \left \lceil \frac{g}{2} \right \rceil \right)^2 - 24\left \lceil \frac{g}{2} \right \rceil + 8)^2 < 7^4 \left( \left \lceil \frac{g}{2} \right \rceil \right)^4 - 2,$ where $\mu$ is the Perron--Frobenius eigenvalue of $NN^{T}$ as described in Theorem \ref{thm:ThurCons}. Thus, \[\log(\lambda(f)) \leq \log(\mu + 2) \leq \log \left(7^4 \left( \left \lceil \frac{g}{2} \right \rceil \right)^4 \right) = 4 \log \left( \left \lceil \frac{g}{2} \right \rceil \right) + 4 \log(7).\]\end{proof}

\begin{proof}[Case 2] Now consider the case when $4 \leq n < 2g$. We will illustrate our construction in Figure \ref{fig:4punctures} in the case of a genus $4$ surface. We will build our pair of filling multicurves on $S_{g,n}$ in the following way. We will divide $S_g$ into a sphere with $\lfloor \frac{n}{2} \rfloor$ holes and subsurfaces of genus at most $\left \lceil \frac{2g}{n} \right \rceil$ each with a single boundary component. Puncture each subsurface once, and as before, we fill each of these subsurfaces with an Aougab--Taylor pair $\alpha$ and $\beta$, shown in red and blue, respectively, in Figure \ref{fig:4punctures}. We then add an additional puncture to each subsurface so that it is near the boundary component of that subsurface. This is illustrated in Figure \ref{fig:4punctures}. Let $A$ be the union of the $\alpha$ curves and $B$ be the union of the $\beta$ curves from our Aougab--Taylor pairs. Now view the subsurface as being arranged cyclically, as shown in Figure \ref{fig:4punctures}, and for consecutive pairs of punctures, one coming from the Aougab--Taylor pair and one a puncture added near the subsurface boundary, add a red curve to our multicurve $A$ which bounds a twice punctured disc. We have now constructed a pair of filling multicurves $A$ and $B$ which fill our surface $S_{g,n}$. 

Note that these additional bounding pair curves will each intersect with two blue curves. They will intersect with one blue curve twice and with the other blue curve at most $24 \left( \left \lceil \frac{2g}{n} \right \rceil \right)^2 - 24\left \lceil \frac{2g}{n} \right \rceil + 8$ times. The picture on the left of Figure \ref{fig:4punctures} illustrates the case of an even number of punctures and the picture on the right the case of an odd number of punctures. 

\begin{figure}[H]
\centering
\includegraphics[width = 2.25in]{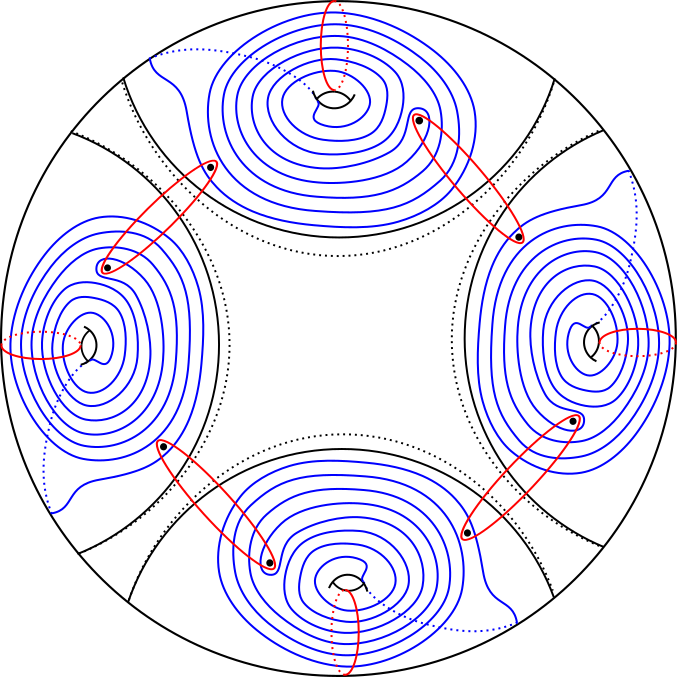}
\hspace{5mm}
\includegraphics[width = 2.25in]{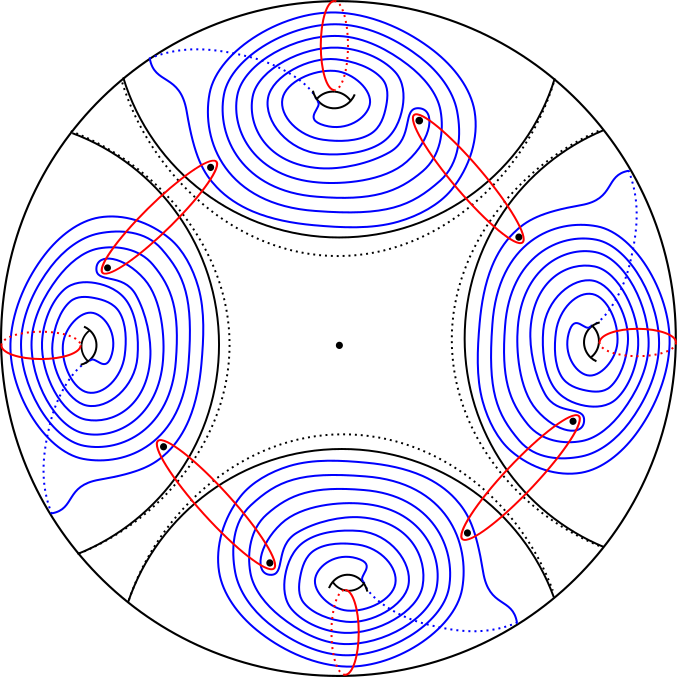}
\caption{Examples of filling multicurves, $A$ and $B$, for $4 \leq n < 2g$}
\label{fig:4punctures}
\end{figure}

Let $f = T_{A}T_{B}^{-1}$. Note that $f$ is a pseudo-Anosov pure braid for the same reasons given in Case 1. Thus, we can proceed immediately to computing the maximum row sum of $NN^{T}$ in order to bound $\lambda(f)$. We can compute the maximum row sum of $NN^T$ by considering the labeled bipartite graph in Figure \ref{fig:adjGraph1} that describes the intersection pattern of red and blue curves.

\begin{figure}[H]
\centering
\includegraphics[width = 1.85in]{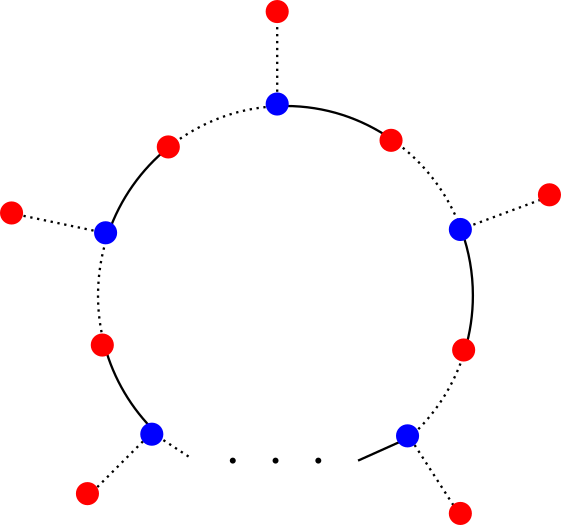}
\caption{Bipartite graph for $A$ and $B$ when $4 \leq n < 2g$}
\label{fig:adjGraph1}
\end{figure}

Note that each blue vertex has valence $3$ and each red vertex has valence at most $2$. Furthermore, the dashed edges have label at most $24 \left( \left \lceil \frac{2g}{n} \right \rceil \right)^2 - 24\left \lceil \frac{2g}{n} \right \rceil + 8$ and the solid edges have label $2$. 

Thus, for the red vertices of valence $2$ we have a corresponding row sum of at most \[2\left( 24 \left( \left \lceil \frac{2g}{n} \right \rceil \right)^2 - 24\left \lceil \frac{2g}{n} \right \rceil + 8 \right)^2 + 6\left( 24 \left( \left \lceil \frac{2g}{n} \right \rceil \right)^2 - 24\left \lceil \frac{2g}{n} \right \rceil + 8 \right) + 4.\]

For the red vertices of valence $1$ we have a corresponding row sum of at most \[2 \left( 24 \left( \left \lceil \frac{2g}{n} \right \rceil \right)^2 - 24\left \lceil \frac{2g}{n} \right \rceil + 8 \right)^2 + 2\left( 24 \left( \left \lceil \frac{2g}{n} \right \rceil \right)^2 - 24\left \lceil \frac{2g}{n} \right \rceil + 8 \right).\]

Note that each of these is at most $1152 \left( \left \lceil \frac{2g}{n} \right \rceil \right)^4 - 2 < 6^4 \left( \left \lceil \frac{2g}{n} \right \rceil \right)^4 - 2.$ Thus, the maximum row sum of $NN^T$ is bounded above by $6^4 \left( \left \lceil \frac{2g}{n} \right \rceil \right)^4 - 2$ and we have that \[\log (\lambda(f)) \leq \log (\mu + 2) \leq \log \left(6^4 \left( \left \lceil \frac{2g}{n} \right \rceil \right)^4 \right) = 4 \log \left( \left \lceil \frac{2g}{n} \right \rceil \right) + 4 \log(6).\qedhere\]\end{proof}

\begin{proof}[Case 3]Note that when $n \geq 2g$ the inequality in Theorem \ref{thm:upperBound} implies that we have a constant upper bound on $L(\PB_n(S_g))$. The construction given above is for $n < 2g$, but can be extended to give a constant upper bound as we add additional punctures. Suppose we have $n \geq 2g$. We can divide $S_g$ into subsurfaces of genus 1. We then puncture each of these subsurfaces and fill each one with an Aougab--Taylor pair, $\{ \alpha, \tau \}$, such that $i(\alpha, \tau) = 6$ using the construction in Section \ref{section:pointpushing} and continue to add punctures to the ``central" portion of $S_{g,n}$ as shown in Figure \ref{fig:constantBound} where the red curves belong to $A$ and the blue curves belong to $B$. Note that this manner of adding additional punctures does not increase the number of pairwise intersections between red and blue curves nor does it introduce any curves that have nonzero intersection with more than two other curves. 

\begin{figure}[H]
\centering
\includegraphics[width = 2.25in]{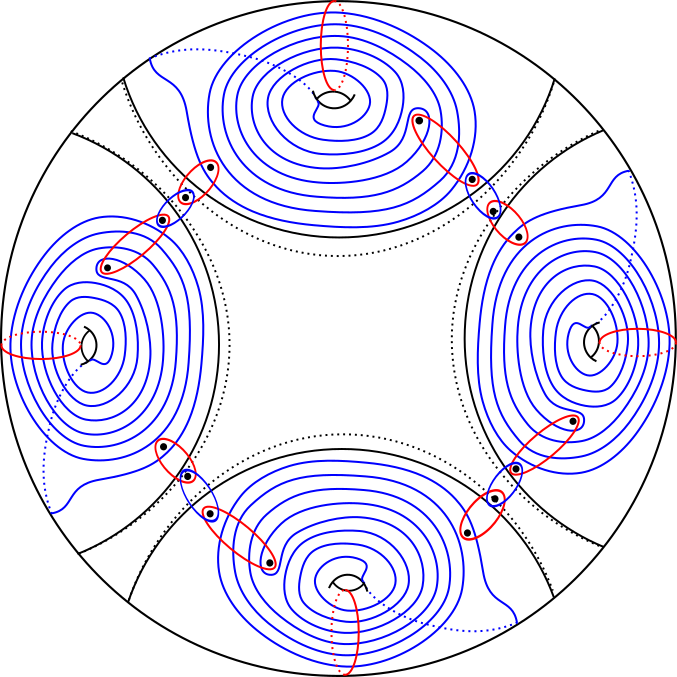}
\hspace{5mm}
\includegraphics[width = 2.25in]{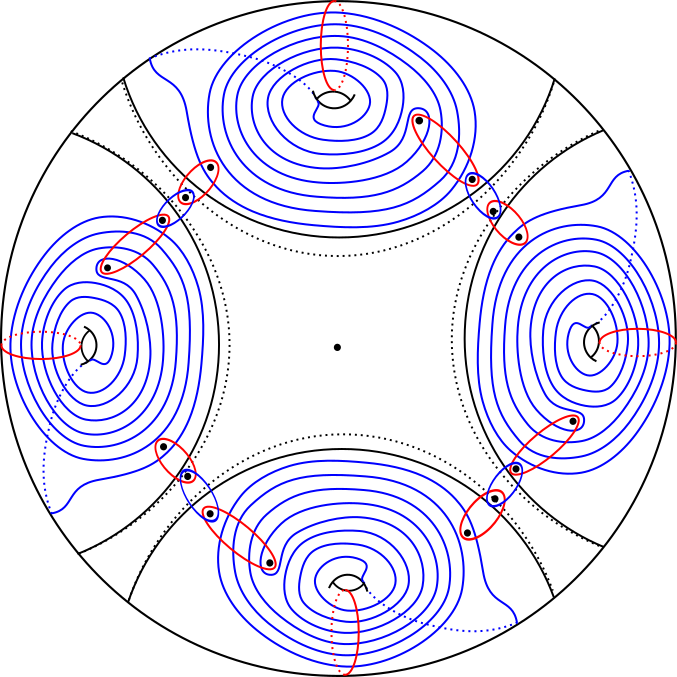}
\caption{Examples of filling multicurves, $A$ and $B$, for $n \geq 2g$}
\label{fig:constantBound}
\end{figure}

Let $f = T_{A}T_{B}^{-1}$. Note that $f$ is a pseudo-Anosov pure braid by the same reasoning used previously. Thus, just as we did before, we can proceed directly to computing the maximum row sum of $NN^{T}$ in order to bound $\lambda(f)$. We can compute the maximum row sum of $NN^T$ by considering the labeled bipartite graph in Figure \ref{fig:adjGraph2} which is constructed in the same way as the bipartite graph in Figure \ref{fig:adjGraph1}. 

\begin{figure}[H]
\centering
\includegraphics[width = 1.85in]{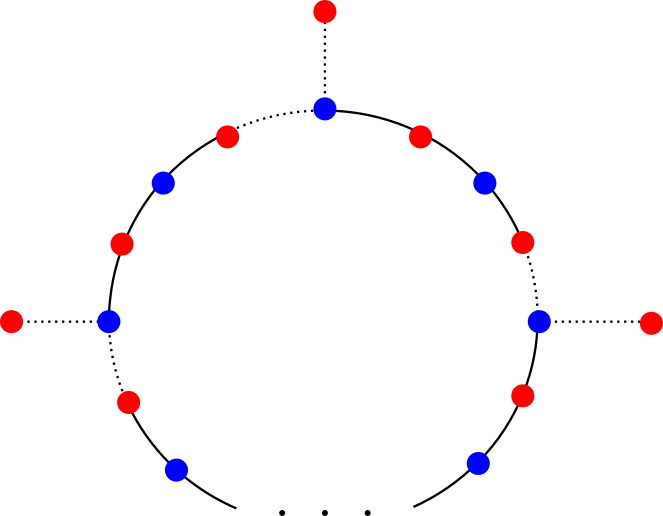}
\caption{Bipartite graph for $A$ and $B$ when $n > 2g$}
\label{fig:adjGraph2}
\end{figure}

The dashed edges in Figure \ref{fig:adjGraph2} are labeled by 8 and the solid edges are labeled by 2. Thus, we can compute that the maximum row sum of $NN^T$ is 152 and we have that $\log(\lambda(f)) <  4 \log(6).$ \end{proof}

Thus, we have addressed each of our three cases and shown that \[L(\PB_n(S_g)) \leq 4 \log \left( \left \lceil \frac{2g}{n} \right \rceil \right) + 4 \log(7),\] as desired.\end{proof}

\section{A Constant Lower Bound}
\label{section:lowerbound}

In this section we provide a constant lower bound on $L(\PB_n(S_g)$. 

\begin{thm}\label{thm:lowerbound} For a surface $S_{g}$ of genus $g \geq 2$ with $n \geq 1$, we have \[.000155 \leq L(\PB_n(S_g)).\]\end{thm}

The proof of Theorem \ref{thm:lowerbound} relies on the following result of Agol--Leininger--Margalit which can be found in \cite{ALM16}. 

\begin{prop} \label{prop:AgolLeinMarg} Let $S$ be a surface and $f \in \Mod(S)$ pseudo-Anosov, then \[.00031 \left(\frac{\kappa(f) + 1}{|\chi(S)|} \right) \leq \log(\lambda(f)),\] where $\kappa(f)$ is the dimension of the subspace of $H_1(S; \mathbb R)$ fixed by $f.$ \end{prop}

In order to make use of this result we must examine the action of a pure surface braid $f \in \PB_n(S_g)$ on $H_1(S_{g,n}; \mathbb R)$. We can place the following lower bound on $\kappa(f)$.

\begin{lemma} \label{lemma:kappa} If $f \in \PB_n(S_g)$, then \[\max \{2g, n-1\} \leq \kappa(f).\]\end{lemma}

\begin{proof}[Proof of Lemma \ref{lemma:kappa}] Let $M_f$ denote the mapping torus of $f$ and let $b_1(M_f)$ denote the first Betti number of $M_f$ with coefficients in $\mathbb R$. Note that $b_1(M_f) = \kappa(f) + 1$. This can be obtained by an application of the Mayer--Vietoris long exact sequence. 

For $n-1 < 2g$, we will show that $b_1(M_f) \geq 2g + 1$. Since $\widehat{f}: S_g \to S_g$, obtained by filling in the punctures of $S_{g,n}$ and extending $f$ to $S_g$, is isotopic to the identity, then $M_{\widehat{f}} \cong M_{\id} \cong S_g \times S^{1}.$ Thus, there exists a map from $M_f \to S_g \times S^{1}$ that induces a surjection on the fundamental groups. By the Hurewicz Theorem, we know that $H_1(M_f; \mathbb Z)$ is isomorphic to the abelianization of $\pi_1(M_f)$. Thus, we have that $\dim (H_1(M_f; \mathbb R)) \geq \rank (\pi_1(S_g \times S^1)^{ab}) = 2g + 1$. Thus, $\kappa(f) \geq 2g$. 

For $2g \leq n-1$, observe that $f$ fixes the subspace, $\mathcal P$, of $H_1(S_{g,n})$ generated by the peripheral curves bounding each puncture because $f$ fixes each puncture. Thus, since $\mathcal P$ has dimension $n-1$, then $\kappa(f) \geq n-1.$ \end{proof}

\begin{proof}[Proof of Theorem \ref{thm:lowerbound}] By Lemma \ref{lemma:kappa}, for a pseudo-Anosov $f \in \PB_n(S_g)$, we have that $\displaystyle \frac{\kappa(f) +1}{|\chi(S_{g,n})|} > \frac12$. This, together with Proposition \ref{prop:AgolLeinMarg}, gives our desired lower bound \[.000155 \leq L(PB_n (S_g)). \qedhere \] \end{proof}

\section{A Lower Bound for Fixed Number of Punctures}
\label{section:anotherlowerbound}

We conclude with a proof of the lower bound which, for fixed $n$, goes to infinity as $g$ does. 

\begin{thm} \label{thm:infty}  If $f \in \PB_n(S_g)$ is pseudo-Anosov and $g > 5$, then \[\frac{1}{3} \log\left(1 + \frac{\log \left(\frac{g-2}{3}\right) + 2}{160n}\right) \leq \log(\lambda(f)).\] \end{thm}

\begin{proof}[Proof of Theorem \ref{thm:infty}] By Theorem \ref{thm:canonical}, we have a hyperbolic/conformal structure on $S_g$ and an isotopy $F_t$ through quasiconformal maps from the identity to $f$ such that for each $t$ the quasiconformal constant, $K_t$, satisfies \[\log(K_t) \leq 3 \log(\lambda(f)).\] Choose a lift, $\widetilde{F_t}$, of $F_t$ to the universal cover, $\mathbb H^2$, of $S_g$ so that $\widetilde{F_0}$ is the identity. Therefore, $\widetilde{F_t}$ is the identity on the circle at infinity. Thus, we can apply Theorem \ref{thm:gehring} and Theorem \ref{thm:canonical} to see that \[\varkappa \left(\max_{x \in \mathbb H^2} \rho(x, \widetilde{F_t}(x)) \right) \leq \frac{1}{2} \log(K_t)  \leq \frac{3}{2} \log(\lambda(f)).\] 
Since this holds for all $t \in [0,1],$ we have \[\varkappa \left(\max_{t \in [0,1]} \max_{x \in \mathbb H^2} \rho(x, \widetilde{F_t}(x))\right) \leq \frac{3}{2} \log(\lambda(f)).\] 
Note that when measuring distance on the surface we are using the hyperbolic metric, denoted $d_{S_g}$, and in the hyperbolic plane we are using the Poincar\`e metric, denoted $\rho$, which is one-half the hyperbolic metric. Thus, the covering map $\pi:\mathbb H^2 \to S_g$ is $2$-Lipschitz and for all $x \in \mathbb H^2$, \[d_{S_g}(\pi(x), F_t(\pi(x))) \leq 2 \rho(x, \widetilde{F_t}(x)).\] 
So we have that \[\varkappa \left(\max_{t \in [0,1]} \max_{x \in S_g} d_{S_g}(x, F_t(x)) \right) \leq 3 \log(\lambda(f)).\] 
If $\{z_1, \ldots, z_n\}$ are the marked points of $S_{g}$ such that $S_{g,n} = S_g \setminus \{z_1, \ldots, z_n\}$, then for each $i$, $\gamma_i: t \mapsto F_t(z_i)$, with $t \in [0,1]$, is a closed curve. Since $f$ is pseudo-Anosov, $\gamma_1 \cup \cdots \cup \gamma_n$ fills $S_g$. These $n$ curves define the $1$-skeleton, $\Gamma$, of a cell decomposition of $S_g$. 
Thus, for some $i$, \[\frac{\diam(\Gamma)}{n} \leq 2 \max_{t \in [0,1]} d_{S_g}(z_i, F_t(z_i)).\] 
By Theorem \ref{thm:appendix}, $\frac{\log \left(\frac{g-2}{3}\right)- 2}{40n}\leq \frac{\diam(\Gamma)}{n}.$ 
By Theorem \ref{thm:gehring}, $\varkappa$ is strictly increasing, so we have that \[\varkappa \left(\frac{\log \left(\frac{g-2}{3}\right)-2}{80n} \right) \leq \varkappa \left(\frac{\diam(\Gamma)}{2n} \right) \leq \varkappa \left(\max_{t \in [0,1]} d_{S_g}(z_i, F_t(z_i)) \right) \leq 3 \log(\lambda(f)).\] 
Since, by Theorem \ref{thm:gehring}, $\log \left(1 + \frac{\log \left(\frac{g-2}{3}\right) - 2}{160n} \right) \leq \varkappa \left( \frac{\log \left(\frac{g-2}{3}\right)- 2}{80n} \right)$, then we have that \[\frac{1}{3}\log \left(1 + \frac{\log \left(\frac{g-2}{3}\right) - 2}{160n} \right) \leq \log(\lambda(f)),\] as desired. \end{proof}

\appendix

\section{Diameter of a Graph Embedded in a Surface}

Let $S$ be a closed genus $g \geq 2$ hyperbolic surface and let $\Gamma$ be the $1$-skeleton of a cell decomposition of $S$. Our goal in this appendix is to provide a lower bound on the diameter of $\Gamma$, which we define as \[\diam(\Gamma) = \max_{x, y \in \Gamma} d_{S}(x,y).\] This lower bound is a crucial piece of the proof of Theorem \ref{thm:infty}. For a result related to Theorem \ref{thm:appendix}, see \cite{martelli2017}.

\begin{thm} \label{thm:appendix} Let $\Gamma$ be an embedded graph in $S$ such that $S \setminus \Gamma$ is a collection of disks. If $g > 5$, then \[\frac{\log \left(\frac{g-2}{3}\right) - 2}{40} \leq \diam(\Gamma).\] \end{thm}

The first ingredient we will need for the proof of Theorem \ref{thm:appendix} is a type of generalized triangulation of  $S$ which consists of both geodesic triangles and a type of annular generalization of a triangle called a {\bf trigon} as defined by Buser; see \cite{buser92}. 

\begin{mydef} Let $S$ be a compact Riemann surface of genus $\geq 2$. A closed domain $D \subset S$ is called a {\bf trigon} if it is a simply connected, embedded geodesic triangle or if it is a doubly connected, embedded domain, with one boundary component a smooth closed geodesic and the other boundary component two geodesic arcs as shown in Figure \ref{fig:trigon}. The closed geodesic and the two arcs are the {\bf sides} of $D$.\end{mydef} 

\begin{figure}[H]
\centering
\includegraphics[width=1.25in]{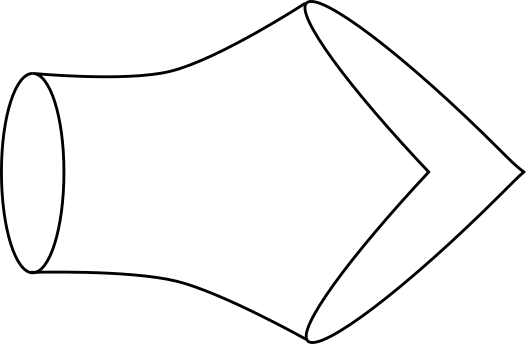}
\caption{A trigon}
\label{fig:trigon}
\end{figure}

Buser proved that $S$ admits such a triangulation into trigons of controlled size. 

\begin{thm}[Buser \cite{buser92} Theorem 4.5.2] \label{thm:buser} Any compact Riemann surface of genus $\geq 2$ admits a triangulation such that all trigons have sides of length $\leq \log 4$ and area between $0.19$ and $1.36$. Furthemore, all geodesic triangles have sides of length at least $\log(2).$\end{thm}

Suppose we have a generalized triangulation $T$ of $S$ as in Theorem \ref{thm:buser}. We will extend our generalized triangulation to an even more general combinatorial model, $T'$, for $S$ in the following way. First, we note that a computation (which we omit) using equation (iii) of Theorem 2.3.1 in \cite{buser92} shows that the width (i.e.~minimal distance between non-adjacent boundary components) of a doubly connected trigon which occurs in $T$ is at least $\frac14$.  Next, consider collars of closed geodesics in $S_g$ formed by gluing together two doubly connected trigons along their closed geodesic sides as in Figure \ref{fig:collar}.  Now we divide each collar along appropriately chosen simple closed curves (each an equidistant-curve to the closed geodesic) into annuli between simple closed curves and two {\bf generalized trigons} on the ends, so that each annulus or generalized trigon has width between $\frac14$ and $\log(2) > \frac12$; see the right-hand side of Figure \ref{fig:collar}. Our combinatorial model $T'$ consists of three types of {\bf pieces}: geodesic triangles, generalized trigons, and annuli. Note that each of these pieces is of bounded size. 

\begin{figure}[H]
\centering
\includegraphics[width = 2.5in]{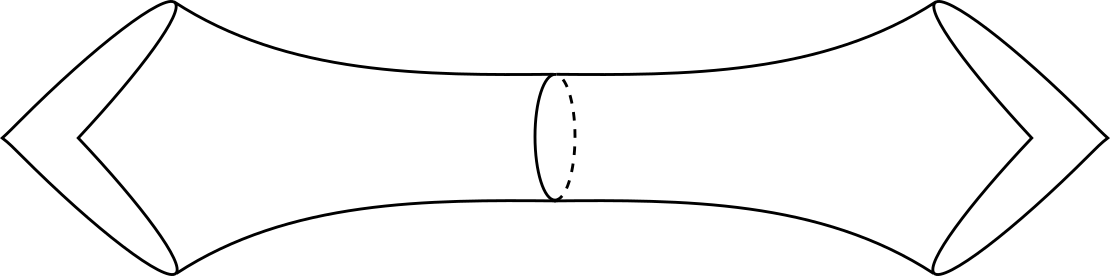}
\hspace{5mm}
\includegraphics[width = 2.5in]{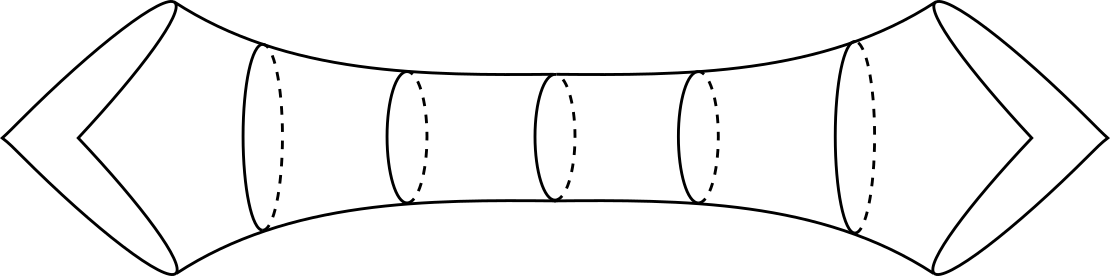}
\caption{A collar formed by two trigons}
\label{fig:collar}
\end{figure}

We can now define the combinatorial length of a geodesic between two points $p, q \in S$ in terms of our combinatorial model $T'$. For a geodesic segment $\overline{pq} \subset S$ between $p$ and $q$ we define the combinatorial length of $\overline{pq}$, denoted by $\ell_{C}(\overline{pq})$, as the minimum number of pieces of $T'$ that $\overline{pq}$ passes through. The following lemma establishes an explicit inequality between $\ell_{C}$ and the hyperbolic length $\ell_{S}$.

\begin{lemma} \label{lemma:combinatorial} Let $p, q \in S_g$, let $\overline{pq}$ be a geodesic segment between them, and let $T'$ be the extended combinatorial model of $S$ given above. Then $\ell_C(\overline{pq}) \leq 40 \cdot \ell_S(\overline{pq}) +2$. \end{lemma}

\begin{proof}[Proof of Lemma \ref{lemma:combinatorial}] Note that $\overline{pq}$ can be subdivided into segments which each lie inside a single piece of $T'$. Our proof of Lemma \ref{lemma:combinatorial} will consist mainly of analyzing which segments of $\overline{pq}$ are {\bf short} and which are {\bf good}. We will then show that segments of $\overline{pq}$ cannot be short too many times in a row. 

There are three types of short segments we will consider, one in each of the three types of pieces. In order to define the first type, we add midpoints to each side of the geodesic triangles in $T$. A segment which has endpoints on adjacent subdivided pieces of a single geodesic triangle is called {\bf short}. The second type of short segment occurs when $\overline{pq}$ enters and exits an annulus from a single side instead of passing through the entire width of the annulus. In this situation, a segment which has both endpoints on a single boundary component of an annulus will also be considered {\bf short}. The third type of {\bf short} segment occurs when a segment without self intersections has endpoints on adjacent subdivided pieces of the geodesic boundary arcs of a generalized trigon, cutting off a corner, as shown by the blue segment in Figure \ref{fig:shortSegment}. If a segment is not short, then we will call it {\bf good}. 

\begin{figure}[H]
\centering
\includegraphics[width=1.25in]{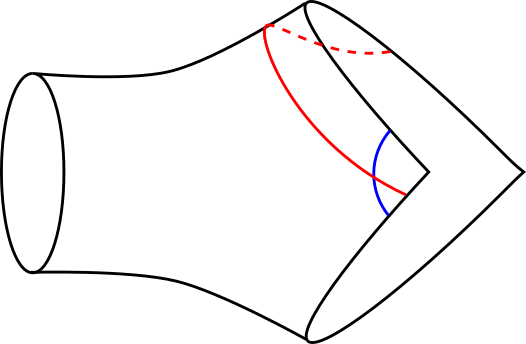}
\caption{Short (blue) and good (red) segments in a generalized trigon}
\label{fig:shortSegment}
\end{figure}

Recall the following formula for a geodesic triangle in the hyperbolic plane where $a, b, c$ are the sides of the triangle and $\alpha, \beta, \gamma$ are the respective opposite angles: \begin{equation} \label{eq:trig} \cos(\gamma) = \frac{\cosh(c) - \cosh(a)\cosh(b)}{-\sinh(a)\sinh(b)}.\end{equation}
We can find a lower bound on the length of $\gamma$ for a geodesic triangle in $T'$ by maximizing the length of $a$ and $b$ and minimizing the length of $c$. Taking $a, b = \log(4)$ and $c = \log(2)$, equation (\ref{eq:trig}) implies that $\gamma > \frac{\pi}{9}$. Thus, there is a lower bound of $\frac{\pi}{9}$ on the interior angles of the triangles in $T$. So we conclude that $\overline{pq}$ has no more than $\frac{\pi}{\nicefrac{\pi}{9}} = 9$ short segments of the first type in a row. Next, we note that there cannot be two short segments of type two or three in a row. So we can assume that for every ten adjacent segments of $\overline{pq}$, at least one of them is good.

We now establish that good segments have length at least $\frac{1}{4}$. Once again we have three types of segments to consider, the shortest possible good segments within each of our three types of pieces in $T'$. Within a geodesic triangle in $T'$ the shortest possible good segment is one that joins the midpoints of two sides of a triangle. Once again using equation (\ref{eq:trig}), we see that the length of a good segment is bounded below by \[\cosh^{-1}\left(-\sinh(\log(2))\sinh(\log(2)) \cos \left(\frac{\pi}{9}\right) + \cosh(\log(2))\cosh(\log(2))\right) \geq \frac{1}{4}.\] Within a generalized trigon the shortest possible good segment is a perpendicular segment going from the closed boundary component of the trigon to the midpoint of one of the geodesic arc boundary components. A segment of this type has length at least $\frac{1}{4}$ by our definition of $T'$. One might think that a shorter possible good segment in a generalized trigon is one passing from one geodesic arc boundary to the other as shown by the red arc in Figure \ref{fig:shortSegment}. However, this red arc has length at least half of the length of the geodesic arc boundary and so it has length at least $\frac{\log(2)}{2} > \frac14$. Lastly we have that within an annulus the shortest possible good segment is a perpendicular segment passing from one boundary component to the other, which has length at least $\frac14$ since we defined our annuli to have width at least $\frac{1}{4}$.

Thus, at worst we have that $\frac{1}{4} \cdot \frac{\ell_C(\overline{pq}) - 2}{10} \leq \ell_S(\overline{pq})$, where the $-2$ comes from the fact that the initial and terminal segments of $\overline{pq}$ can be arbitrarily short depending on where they lie within a piece of $T'$, but still add $2$ to $\ell_{C}(\overline{pq})$. So we have that $\ell_C(\overline{pq}) \leq 40 \cdot \ell_S(\overline{pq})  + 2$, as desired. \end{proof}

We now define the combinatorial distance, denoted $d_{C}$, between two points $p, q \in S_g$ as \[d_C(p,q) = \inf{\{ \ell_C(\overline{pq}) \, : \, \overline{pq} \text{ is a geodesic segment between $p$ and $q$} \} }.\] Thus, by Lemma \ref{lemma:combinatorial}, we have that $d_C(p,q) \leq 40 \cdot d_S(p,q) + 2.$

Let $T_{\Gamma}$ be the subset of $T'$ that minimally covers $\Gamma$, where a piece $t \in T'$ belongs to $T_{\Gamma}$ if $\Gamma \cap t \neq \varnothing$. We will denote by $\Gamma'$ the $1$-skeleton of $T_{\Gamma}$ together with a geodesic arc for each generalized trigon and annulus in $T'$ as shown by the dotted arc in Figure \ref{fig:gammatrigon}, which ensures $\Gamma'$ is connected. 

\begin{figure}[H]
\centering
\includegraphics[width=1.25in]{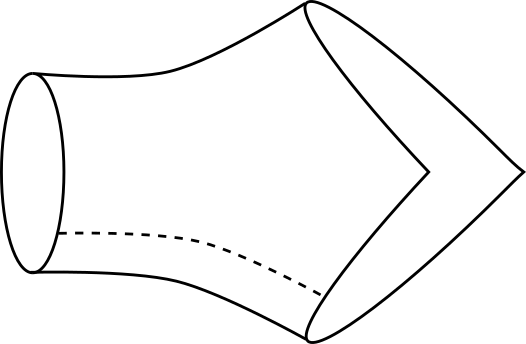}
\caption{A generalized trigon's contribution to $\Gamma'$}
\label{fig:gammatrigon}
\end{figure}

We will need one more fact relating the length of $\Gamma'$ to the genus of our surface before we continue to the proof of our main result.

\begin{lemma} \label{lemma:length} Let $\Gamma'$ be as above. Then \[\ell (\Gamma') =  \sum_{e \in \Gamma'} \ell(e) > 2 \pi (g-1).\] \end{lemma}

\begin{proof}[Proof of Lemma \ref{lemma:length}] Note that if $\alpha$ is a simple closed curve that intersects $\Gamma$, then it must intersect $\Gamma'$. Thus, $\Gamma'$ fills $S_g$, since $\Gamma$ does. Now because $\Gamma'$ fills, it cuts $S_g$ into polygons. The sum of the lengths around these polygons is $2 \ell(\Gamma')$, while the sum of their areas is $4\pi(g-1)$. 

Now recall that the maximum area $A(p)$ enclosed by a loop of length $p$ in the hyperbolic plane is at most the area of a circle of radius $r = \sinh^{-1}\left(\frac{p}{2 \pi} \right).$ Therefore, \[A(p) \leq 4 \pi \sinh^2\left(\frac{\sinh^{-1}\left(\frac{p}{2 \pi}\right)}{2}\right) \leq 4 \pi \sinh^2 \left(\frac{1}{2} \log \left(1 + \frac{p}{\pi} \right)\right) = \frac{p^2}{p + \pi} < p.\] Applying this inequality to each of the polygons and summing, we have $4 \pi (g - 1) \leq 2\ell(\Gamma'),$ as desired \end{proof}

Lemma \ref{lemma:length} implies that $T_{\Gamma}$ contains at least $g-1$ pieces of $T'$, since each piece contributes a length of at most $3 \log(4) > 2 \pi$ to $\ell(\Gamma')$. We are now in a position to prove Theorem \ref{thm:appendix}.

\begin{proof}[Proof of Theorem \ref{thm:appendix}] Let $T'$ and $T_{\Gamma}$ be as described previously. Consider the graph $G$ in $S_g$ which is dual to $T'$, that is the vertices of $G$ each correspond to a piece of $T'$ and edges in $G$ correspond to shared boundary components. Note that each vertex of $G$ has valence at most $3$. Thus, we if we take a base piece $\Delta_0 \in T_{\Gamma} \subset T'$, then we know that at combinatorial distance $d$ from $\Delta_0$ there are at most $3 \cdot 2^{d-1} + 1$ pieces in $T'$. This is because a ball of radius $d$ in $G$ has size at most $3 \cdot 2^{d-1}$. So, unless $g-1 < 3\cdot 2^{d-1} + 1,$ there is a piece of $T_{\Gamma}$ not in this ball. Hence, the combinatorial diameter of $T_{\Gamma}$ (within $T'$) is at least $\log \left(\frac{g-2}{3}\right) < \log_2 \left(\frac{g-2}{3}\right) < \diam_{C}(T_{\Gamma})$ for $g > 5$ and we are done. \end{proof}

\bibliographystyle{plain}

\bibliography{sketchOfSomething}

\end{document}